\documentclass[10pt]{amsart}

\usepackage{amsmath,amsfonts,amssymb, color,cmll}
\usepackage{tikz}
\usepackage{caption} 

\theoremstyle{plain}
\newtheorem{thm}{Theorem}[section]
\newtheorem{lem}[thm]{Lemma}

\newtheorem{prop}[thm]{Proposition}

\newtheorem{conj}[thm]{Conjecture}

\newtheorem{definition}[thm]{Definition}

\newtheorem*{claim*}{Claim}
\newtheorem*{thm*}{Theorem}

\def\R{\mathbb{R}}

\def\H{\mathcal{H}}
\def\dH{\dim_{\mathcal{H}}}

\newcommand{\Span}{\operatorname{Span}}
\begin{document}
\title{On unions of geodesics and projections of invariant sets}
\date{}
\author{Longhui Li}
\begin{abstract}
    Let $M$ be a $d$-dimensional complete Riemannian manifold and let $\pi: SM \to M$ denote the canonical projection from the unit tangent bundle. We prove that if $E \subset SM$ is a set that invariant under the geodesic flow with Hausdorff dimension $\dH E \ge 2(k-1)+1 +\beta$ for some integer $1 \le k \le d-1$ and some $\beta \in [0,1]$, then the projection $\pi(E)$ satisfies $\dim_\H \pi(E) \ge k + \beta$. In other words, this yields a lower bound on the Hausdorff dimension of unions of geodesics in $M$. Our theorem extends a result of J. Zahl concerning unions of lines in $\mathbb{R}^d$. The proof relies on the transversal property of geodesics, an appropriate $(k+1)$-linear curved Kakeya estimate, and the Bourgain–Guth argument.

\end{abstract}
\maketitle

\section{Introduction}
    \subsection{Unions of affine $n$-planes}
    Let $A(d,n)$ denote the space of affine $n$-planes in $\R^d$. Each affine $n$-plane can be uniquely written as $V+a$ where $V$ is a $n$-dimensional subspace and $a\in V^{\bot}$. We equip $A(d,n)$ with the metric
    \begin{align*}
        d(V+a,W+b)=\|\pi_V-\pi_W\|+|a-b|,
    \end{align*}
    where $\|\pi_V-\pi_W\|$ denote the operator norm of $\pi_V-\pi_W$.
    
    D. Oberlin \cite{DMObe14} showed that if $A$ is a Borel set with $\dH\mathcal{A}=\alpha>(n+1)(d-n)-n$, the set $\bigcup_{V\in\mathcal{A}}V$ has positive Lebesgue measure in $\mathbb{R}^d$. This result is sharp for general families of $n$-planes $\mathcal{A}$. 
    D. Oberlin also posed the following conjecture on the Hausdorff dimension of unions of affine lines.
    \begin{conj}\label{conjecture, unions of lines}
        Suppose $1\leq k\leq d-1$ is an integer and let $\beta\in[0,1]$. If $\mathcal{A}\subset A(d,1)$ with $\dH\mathcal{A}\geq 2(k-1)+\beta$, then 
        \begin{align*}
            \dH\bigcup_{l\in\mathcal{A}}l\geq k+\beta.
        \end{align*}
    \end{conj}
    The conjecture is best possible if one takes the set of all lines contained in a $\beta$-dimensional union of parallel affine $k$-planes. R. Oberlin \cite{RObe16} proved the finite field version of this conjecture. Héra, Keleti, and Máthé \cite{HKM19} proved this conjecture for $k=1$, and D. Oberlin \cite{DMObe14} proved this conjecture for $k=d-1$.
    Finally, Zahl \cite{Zah23} proved this conjecture for the remaining values of $k$ using the multilinear Kakeya theorem and the Bourgan-Guth argument.

    The conjecture where affine lines are replaced by affine $n$-planes was posed by Héra \cite{Hera19} and proved by Gan \cite{Gan23union} using the Bramscamp-Lieb inequality.
    \subsection{Unions of geodesics and projections of sets invariant under geodesic flows}
     The natural analogues of affine $n$-planes are  $n$-dimensional totally geodesic submanifolds. Any Riemannian manifold has many $1$-dimensional totally geodesic submanifolds, i.e, geodesics, but not all Riemannian manifolds have $n$-dimensional totally geodesic submanifolds when $n\geq 2$. In this article, we focus on unions of geodesics on complete Riemannian manifolds.

    The first problem is determining the proper measurement of a collection of geodesics. Suppose $M$ is a complete Riemannian manifold and $\Gamma=\{\gamma_i\}_{i\in I}$ is a collection of geodesics. Let $SM$ denote the unit tangent bundle of $M$ and $\pi$ denote the canonical projection from $SM$ to $M$. Suppose $\gamma_i(t)$ is a parametrization of $\gamma_i$ by arc length. We consider the set
    \begin{align}
        S(\Gamma):=\{(\gamma_i(t),\gamma_i'(t)):t\in\R,\;i\in I\}\subset SM.
    \end{align}
    It is not hard to check that $S(\Gamma)$ is a set invariant under the geodesic flow and
    \begin{align}
        \bigcup_{i\in I}\gamma_i=\pi(S(\Gamma)).
    \end{align}
    So in our case, the problem with the size of unions of geodesics is equivalent to the problem with the size of projections of sets invariant under the geodesic flow. For simplicity, we say that $E\subset SM$ is invariant if it is invariant under the geodesic flow.

    In the case $M=\R^d$, suppose $A\subset A(d,1)=\{(x,v)\in\R^d\times S^{d-1}:x\;\bot\; v\}$ and 
    \begin{align*}
        S(A)=\{(x+tv,v)\in\R^d\times S^{d-1}:t\in\R,\;(x,v)\in A\}\subset S(\R^d).
    \end{align*}
    $S(A)$ is an invariant set, and we have
    \begin{align}
        \dH A=\dH S(A)-1.
    \end{align}
    
    Motivated by this, we can state the corresponding theorem on complete manifolds in terms of projections of invariant sets. 
   \begin{thm}\label{Hausdorff dimension, unions of geodesics}
        Suppose 
        $M$ is a $d$-dimensional complete Riemannian manifold. Let $1\leq k\leq d-1$ be an integer and $\beta\in[0,1]$. If $E\subset SM$ is an invariant set with $\dH E\geq 2(k-1)+1+\beta$, then 
        \begin{align*}
            \dH\pi(E)\geq k+\beta.
        \end{align*}
    \end{thm}

    In Section 2, we will see that the Hausdorff dimension on $SM$ can be defined independently of the explicit Riemannian metrics. 
    
    Note that when $d=2$, Theorem \ref{Hausdorff dimension, unions of geodesics} will imply that the projection $\pi$ preserves Hausdorff dimension, i.e, if $\dH E\leq 2$, then
    \begin{align*}
        \dH \pi(E)=\dH E.
    \end{align*}
    
    The case $d=2$ of Theorem \ref{Hausdorff dimension, unions of geodesics} was proved by Ledrappier and Lindenstrauss \cite{LL03} using simple geometry properties and the energy method. Although their original result concerns projections of invariant measures on compact manifolds, after some modifications, it also applies to projections of invariant sets on complete manifolds.

    In this paper, we shall prove Theorem \ref{Hausdorff dimension, unions of geodesics} for all $d\geq 2$.
    With some parametrizations of $SM$, it suffices to consider unions of smooth curves satisfying the regular condition and the transversal condition.
    \begin{thm}\label{Unions of smooth curves}
        Suppose $A\subset[-1,1]^{d-1}\times[-1,1]^{d-1}$ and $\dH A\geq 2(k-1)+\beta$ for some integer $1\leq k\leq d-1$ and $\beta\in[0,1]$. Let $\{\gamma_y\}_{y\in A}$ be a collection of smooth curves such that $\gamma_y$ joints $(y_1,-1)$ and $(y_2,1)$ and can be parametrized by its last coordinate
        \begin{align*}
            \gamma_y(c)=(P_y(c),c),\;c\in[-1,1].
        \end{align*}
        We assume that $\{\gamma_y\}_{y\in A}$ satisfies the followings.
       
       1. (\textbf{Regular Condition}) 
       
       There exists $C\geq 1$ such that
       \begin{align}
           \sup_{y\in A}\|P_{y}(\cdot)\|_{C^2([-1,1])}\leq C.
       \end{align}

        2. (\textbf{Transversal Condition})
        
        Let 
        \begin{align}
            e_y(c)=\frac{\frac{d}{dc}\gamma_y(c)}{|\frac{d}{dc}\gamma_y(c)|}
        \end{align} 
        denote the tangent direction of $\gamma_y$ at the point $\gamma_y(c)$.
        Then there exists $0<m<\infty$ such that for $y,y'\in[-1,1]^{d-1}\times[-1,1]^{d-1}$, $c,c'\in[-1,1]$,
        \begin{align}
           |P_y(c)-P_{y'}(c)|+|e_y(c)-e_{y'}(c)|\geq m(|P_y(c')-P_{y'}(c')|+|e_y(c')-e_{y'}(c')|).
        \end{align}

        Then
        \begin{align}
            \dH( \bigcup_{y\in A}\gamma_y)\geq k+\beta.
        \end{align}
    \end{thm}
    Moreover, Theorem \ref{Hausdorff dimension, unions of geodesics} is sharp for complete manifolds with constant sectional curvature and
    we believe that Theorem \ref{Hausdorff dimension, unions of geodesics} can be improved for complete manifolds with non-constant sectional curvature and dimension $\geq 3$.
    
    \subsection{ Structure of the paper}The paper is organized as follows. In Section 2, we shall give some preliminaries, including the Hausdorff dimension on manifolds and the Hausdorff dimension of invariant sets on $S(\R^d)$. In Section 3, we shall prove Theorem \ref{Hausdorff dimension, unions of geodesics} by Theorem \ref{Unions of smooth curves} using a proper parametrization of $SM$ and the transversal property of geodesics. In Section 4, we shall give a complete proof of Theorem \ref{Unions of smooth curves} by a multilinear curved Kakeya estimate and the Bourgain-Guth argument. In Section 5, we shall give some remarks on Theorem \ref{Hausdorff dimension, unions of geodesics}.

    \subsection{Acknowledgements} The author would like to thank his supervisor, Bochen Liu, for many useful suggestions. The author would also like to thank Josh Zahl and Yakun Xi for helpful discussions.
    
\section{Preliminaries}
    \subsection{Hausdorff dimension on manifolds}
    Suppose $g$ is a Riemannian metric on a manifold $M$ and $d_g$ is the induced Riemannian distance. It is natural to define the Hausdorff dimension associated with $d_g$, denoted by $\dim_{g}$. 
    
    We say that a coordinate chart $(U,\phi)$ is regular if there is another chart coordinate $(V,\psi)$ such that $\overline{U}\subset V$. Suppose $\{(U_i,\phi_i)\}$ is a collection of regular coordinate charts such that $M\subset \cup_i U_i$. By the properties of Hausdorff dimension, we have 
    \begin{align*}
        \dim_g(E)=\sup\limits_{i}\dim_{g}(E\cap U_i).
    \end{align*}
    By Lemma 13.28 in \cite{smoothmanifoldsJohnMlee}, on each $U_i$, the Riemannian metric $g$ is equivalent to the Euclidean metric. So
    \begin{align*}
        \dim_{g}(E\cap U_i)=\dH(\phi_i(E\cap U_i)),
    \end{align*}
    where $\dH$ is the Hausdorff dimension defined in terms of Euclidean distance.
    
    The right-hand side of this equality is independent of $g$, so we can define the Hausdorff dimension that is independent of $g$ (also, independent of the collection $\{(U_i,\phi_i)\}$):
    \begin{align}\label{Hausdorff dimension on manifolds. local}
        \dim_\H(E):=\sup\limits_{i}\dH(\phi_i(E\cap U_i)).
    \end{align}
    We use $\dim_\H$ on both sides, as it will not confuse.
    \subsection{Hausdorff dimension of invariant sets on $S(\R^d)$}
     Let $A(d,1)$ denote the space of affine lines in $\R^d$. Any affine line can be uniquely determined by a direction and a translation orthogonal to this direction. In other words, we can write
     \begin{align*}
         A(d,1)=\{(x,v)\in\R^d\times S^{d-1}:x\;\bot\; v\}\subset\R^d\times S^{d-1}.
     \end{align*}
     We equip $A(d,1)$ with the metric
    \begin{align*}
        d((x,v),(y,w))=|x-y|+|v-w|,\;(x,v),\;(y,w)\in A(d,1).
    \end{align*}
    Suppose $A\subset A(d,1)$ and 
    \begin{align*}
        S(A)=\{(x+tv,v)\in\R^d\times S^{d-1}:t\in\R,\;(x,v)\in A\}\subset S(\R^d)=\R^d\times S^{d-1}.
    \end{align*}
    $S(A)$ is an invariant set, and we have the following.
    \begin{prop}
        \label{Euclidean case}
        \begin{align}
        \dH A=\dH S(A)-1.
    \end{align}
    \end{prop}

    \begin{proof}
        First, we show that
    \begin{align*}
        \dH S(A)\leq \dH A +1.
    \end{align*}
    For $R>0$, we define 
    \begin{align*}
        A_R=\{(x,v)\in A:|x|\leq R\}.
    \end{align*}
    Suppose $\alpha>\dH(A_R)$ and $\{B_i\}_{i\in I}=\{B((x_i,v_i),r_i)\}_{i\in I}$ is a covering of $A_R$ with $\sum_{i\in I}r_i^{\alpha}<\infty$,. For $(x,v)\in A$, we define 
    \begin{align*}
        l_{x,v,R}=\{(x+tv,v):t\in\R,\;|x+tv|\leq R\}.
    \end{align*}
    If $|(x,v)-(x_i,v_i)|<r_i$, it is not difficult to check that $l_{x,v,R}$ is contained in the $(2R+2)r_i$ neighborhood of $l_{x_i,v_i,R}$. The $(2R+2)r_i$ neighborhood of $l_{x_i,v_i,R}$ can be covered by at most $Cr_i^{-1}$ balls of radius $(2R+2)r_i$, thus we have $\H^{\alpha+1}(S(A_R))\lesssim_R1$. Let $R\to\infty$ and $\alpha\searrow\dH A$, we get
    \begin{align*}
        \dH S(A)\leq\dH A+1.
    \end{align*}

    Conversely, for $(x,v)\in\R^d\times S^{d-1}$, we define the map 
    \begin{align*}
        \phi:\R^d\times S^{d-1}\to\R^d\times S^{d-1},\;\phi(x,v)=(x',v),
    \end{align*}
    where $x'=x-\langle x,v\rangle v$. By definition, we have $A=\phi(S(A))$ and $S(A)=\phi^{-1}(A)$.
    
    By direct computation, for $(x,v),(y,w)\in B(0,R)\times S^{d-1}$, we have
    \begin{align*}
        &|\phi(x,v)-\phi(y,w)|\\
        =&|x'-y'|+|v-w|\\
        \leq&|x-y|+|v-w|+|\langle x,v\rangle v-\langle y,w\rangle w|\\
        \leq&|x-y|+|v-w|+|\langle x,v\rangle-\langle y,w\rangle|+|\langle y,w\rangle||v-w|\\
        \leq&|x-y|+|v-w|+|\langle x,v-w\rangle|+|\langle x-y,w\rangle|+|\langle y,w\rangle||v-w|\\
        \leq&2|x-y|+(2R+1)|v-w|.
    \end{align*}
    So $\phi$ is Lipschitz on $B(0,R)\times S^{d-1}$. And for $(x',v)\in S(A_R)$, $\mathcal{H}^1(\phi^{-1}(x',v)\cap B(0,R)\times S^{d-1})>0$. By Lemma in \cite{KM75}, we can conclude that 
    \begin{align*}
        \dH S(A_R)\geq\dH A_R+1, \;R>0.
    \end{align*}
    Let $R\to\infty$, and we have
    \begin{align*}
        \dH S(A)\geq\dH A+1.
    \end{align*}
    \end{proof}

\section{Theorem \ref{Unions of smooth curves} implies Theorem \ref{Hausdorff dimension, unions of geodesics}}
    In this section, we shall show that Theorem \ref{Unions of smooth curves} implies Theorem \ref{Hausdorff dimension, unions of geodesics}.
    
    First, we introduce a parametrization of $SM$.
    We learned this parametrization from \cite{LL03}.
    
    Suppose $p\in M$. We can choose a regular coordinate chart $(U,\phi)$ such that 
    
        (1) $\phi(p)=0$
        
        (2) $[-1,1]^d\subset\phi(U)$
        
        (3) For $p_1\in\mathcal{I}_1=\phi^{-1}([-1,1]^{d-1}\times\{-1\})$, $p_2\in\mathcal{I}_2=\phi^{-1}([-1,1]^{d-1}\times\{1\})$, there exists a unique geodesic $\gamma_{p_1,p_2}\subset U$ jointing them. Moreover, there exists a uniform constant $K$ such that if $\phi(\gamma_{p_1,p_2}(t))=(
        \tilde{x}(t),x^d(t))$, we have 
        \begin{align}\label{nonhorizontal}
            |\frac{d}{dt}\tilde{x}(t)|\leq K|\frac{d}{dt}x^d(t)|.
        \end{align}

    For $p_1\in\mathcal{I}
    _1,p_2\in \mathcal{I}_2$, after reparametrization, we can assume that $|\frac{d}{dt}\gamma_{p_1,p_2}(t)|=1$. Now we can define the parametrization $\psi:D=[-1,1]^{d-1}\times[-1,1]^{d-1}\times[0,1]\to SM$ by
    \begin{align}
        \psi(y_1,y_2,t)=\left(\gamma_{p_1,p_2}\left(d_M(p_1,p_2)t\right),(\frac{d}{dt}\gamma_{p_1,p_2})\left(d_M(p_1,p_2)t\right)\right),
    \end{align}
    where 
    \begin{align*}
        p_1=\phi^{-1}(y_1,-1),\;p_2=\phi^{-1}(y_2,1).
    \end{align*}
    It is clear that $\psi$ is a diffeomorphism from $D$ to $\psi(D)$, then if $E\subset SM$,
    \begin{align}
        \dH(E\cap\psi(D))=\dH(\psi^{-1}(E\cap\psi(D))).
    \end{align}
    
    Let $\pi_1$ denote the projection $(y_1,y_2,t)\mapsto(y_1,y_2)$, $A=\pi_1(\psi^{-1}(E\cap\psi(D)))$ and $\gamma_y(t)=\phi\circ\pi\circ\psi(y,t)$. Then, if furthermore $E$ is invariant, we have
    \begin{align}
        (\phi\circ\pi)(E\cap\psi(D))=\bigcup_{y\in A}\gamma_y.
    \end{align}
    and
    \begin{align}\label{endpoints definition of Hausdorff dimension}
        \dH(A)=\dH (E\cap\psi(D))-1.
    \end{align}

    Now we get a collection of smooth curves $\{\gamma_y\}_{y\in A}$ such that $\gamma_y$ joints $(y_1,-1)$ and $(y_2,1)$. Moreover, with the non-horizontal condition (\ref{nonhorizontal}), we can reparametrize the curve $\gamma_y$ by its last coordinate
    \begin{align}
        \gamma_y(c)=(P_y(c),c),\;c\in[-1,1].
    \end{align}
    And the non-horizontal condition (\ref{nonhorizontal}) implies that 
    \begin{align}
        \sup\limits_{\substack{c\in[-1,1]\\y\in[-1,1]^{d-1}\times[-1,1]^{d-1}}}|\partial ^{\beta}_y\partial ^{\alpha}_cP_{y}(c)|\leq C,\;|\alpha|,|\beta|\leq 2.
    \end{align}
    for some uniform constant $C$.

     It remains to check the transversal property for geodesics. Indeed, we have the following, which is first mentioned in \cite{LL03} to the best of the author's knowledge.
    \begin{lem}\label{Transversal property}
        Suppose $\gamma_y(c)=(P_y(c),c)$, $y\in[-1,1]^{d-1}\times[-1,1]^{d-1}$. Let 
        \begin{align*}
            e_y(c)=\frac{\frac{d}{dc}\gamma_y(c)}{|\frac{d}{dc}\gamma_y(c)|}
        \end{align*} 
        denote the tangent direction of $\gamma_y$ at the point $\gamma_y(c)$.
        Then there exists $0<m<\infty$ such that for $y,y'\in[-1,1]^{d-1}\times[-1,1]^{d-1}$, $c,c'\in[-1,1]$,
        \begin{align}\label{transversal, 2 dim}
           |P_y(c)-P_{y'}(c)|+|e_y(c)-e_{y'}(c)|\geq m(|P_y(c')-P_{y'}(c')|+|e_y(c')-e_{y'}(c')|).
        \end{align}
    \end{lem}
        
    The proof of this lemma is a standard argument in ODEs, and we omit it.

    For any $\epsilon>0$, we can always find $D$ such that 
    \begin{align*}
        \dH (E\cap\psi(D))\geq \dH E-\epsilon.
    \end{align*}
    We apply Theorem \ref{Unions of smooth curves} to $A=\pi_1(\psi^{-1}(E\cap\psi(D)))$ and let $\epsilon\to 0$ and complete the proof.

\section{Proof of Theorem \ref{Unions of smooth curves}}
\subsection{Discretization Argument}
Our first aim is to reduce the proof of Theorem \ref{Unions of smooth curves} to a discretized curved Kakeya estimate. We need some definitions and lemmas.
\begin{definition}[$(\delta,s)$-sets]
   Let $\delta,s>0$, and let $P\subset\R^d$ be a finite $\delta$-separated set. We say that $P$ is a $(\delta,s)$-set if it satisfies 
   \begin{align}
       \#(P\cap B(x,r))\leq (\frac{r}{\delta})^s,\;x\in\R^d,\;r\geq\delta.
   \end{align}
\end{definition}
\begin{lem}[Frostman]\label{Frostman}
    Let $\delta,s>0$, and $E\subset\R^d$ be any set with $\H_{\infty}^{s}(E)=c>0$. Then there exists a $(\delta,s)$-set $P\subset E$ such that $\#P\gtrsim c\delta^{-s}$.
\end{lem}
For a proof of Lemma $\ref{Frostman}$, one can see Appendix A of \cite{FO14}.

Now we are ready to give the discretization argument. Suppose $S=\bigcup_{y\in A}\gamma_y$ and $\dH S=\alpha$. Let $k_0$ be a large positive integer, $\epsilon>0$.
we cover $S$ by a union of balls $\bigcup_{k=k_0}^{\infty}\bigcup_{B\in\mathcal{B}_k}B$, where $\mathcal{B}_k$ is a collection of balls with radius $2^{-k}$, and 
\begin{align}
    \sum_{k=k_0}^{\infty}2^{-k(\alpha+\epsilon)}\#\mathcal{B}_k\leq 1.
\end{align}
In particular, we have 
\begin{align}\label{number of balls}
    \#\mathcal{B}_k\lesssim 2^{k(\alpha+\epsilon)}.
\end{align}
For each $y\in A$, by construction, 
\begin{align*}
\H_{\infty}^1\left(\gamma_y\bigcap\left(\bigcup_{k=k_0}^{\infty}\bigcup_{B\in\mathcal{B}_k}B\right)\right)\gtrsim 1.
\end{align*}
Then there exists $k=k(y)$ such that
\begin{align*}
\H_{\infty}^1\left(\gamma_y\bigcap\left(\bigcup_{B\in\mathcal{B}_k}B\right)\right)\gtrsim \frac{1}{k^2}.
\end{align*}
Let 
\begin{align*}
    A_k=\{y\in A:k(y)=k\},
\end{align*}
then 
\begin{align*}
    \sum_{k=k_0}^{\infty}\H_{\infty}^{2(k-1)+\beta-\epsilon}(A_k)\geq \H_{\infty}^{2(k-1)+\beta-\epsilon}(A)>0.
\end{align*}
Then there exists $k_1\geq k_0$ such that
\begin{align}
    H_{\infty}^{2(k-1)+\beta-\epsilon}(A_{k_1})\gtrsim\frac{1}{k_1^2}.
\end{align}
Let $\delta=2^{-k_1}$. By Lemma \ref{Frostman}, there exists a $(\delta,2(k-1)+\beta-\epsilon)$-set $A'\subset A_{k_1}$ such that
\begin{align}\label{Frostman lemma, delta set}
    \#A'\gtrsim k_1^{-2}\cdot\delta^{-(2(k-1)+\beta-\epsilon)}.
\end{align}
Let 
\begin{align*}
    S'=\bigcup_{B\in\mathcal{B}_{k_1}}2B.
\end{align*}
and let
\begin{align*}
    T_y^{\delta}=\{x=(x',x_d)\in[-1,1]^d:|x'-P_y(x_d)|\leq \delta\}.
\end{align*}
$T_y^{\delta}$ is a curved $\delta$-tube and 
\begin{align}\label{intersection}
    |T_y^{\delta}\cap S'|\gtrsim k_1^{-2}\cdot\delta^{d-1},\;y\in A'.
\end{align}
Then 
\begin{align}
    \|\sum_{y\in A'}\chi_{T_{y}^{\delta}}\|_{L^1(S')}\gtrsim\#A'\cdot k_1^{-2}\cdot\delta^{d-1}\gtrsim k_1^{-4}\cdot \delta^{d-1-(2(k-1)+\beta-\epsilon)}.
\end{align}
On the other hand, by H\"older's inequality, we have 
\begin{align}
    \|\sum_{y\in A'}\chi_{T_{y}^{\delta}}\|_{L^1(S')}\lesssim\|\sum_{y\in A'}\chi_{T_y^{\delta}}\|_{L^p}\cdot|S'|^{1/p'}\lesssim\|\sum_{y\in A'}\chi_{T_y^{\delta}}\|_{L^p}\cdot\delta^{\frac{d-\alpha-\epsilon}{p'}}.
\end{align}
Hence
\begin{align}
    \|\sum_{y\in A'}\chi_{T_y^{\delta}}\|_{L^p}\gtrsim\log(\delta^{-1})^{4}\cdot\delta^{d-2k+1-\beta+\epsilon}\cdot\delta^{-\frac{d-\alpha-\epsilon}{p'}},p\geq 1.
\end{align}
Thus, it suffices to show the following curved Kakeya estimate.
\begin{thm}\label{curved Kakeya, discretized}
     Suppose $A\subset[-1,1]^{d-1}\times[-1,1]^{d-1}$ is a $(\delta,2(k-1)+\beta)$-set. Let $\{\gamma_y\}_{y\in A}$ be a collection of smooth curves such that $\gamma_y$ joints $(y_1,-1)$ and $(y_2,1)$ and can be parametrized by its last coordinate
    \begin{align*}
            \gamma_y(c)=(P_y(c),c),\;c\in[-1,1].
    \end{align*}
    We assume that $\{\gamma_y\}_{y\in A}$ satisfies the followings.
       
       1. (\textbf{Regular condition}) 
       
       There exists $C\geq 1$ such that
       \begin{align}\label{regularity}
           \sup_{y\in A}\|P_{y}(\cdot)\|_{C^2([-1,1])}\leq C.
       \end{align}

        2. (\textbf{Transversal condition})
        
        Let 
        \begin{align}
            e_y(c)=\frac{\frac{d}{dc}\gamma_y(c)}{|\frac{d}{dc}\gamma_y(c)|}
        \end{align} 
        denote the tangent direction of $\gamma_y$ at the point $\gamma_y(c)$.
        Then there exists $0<m<\infty$ such that for $y,y'\in[-1,1]^{d-1}\times[-1,1]^{d-1}$, $c,c'\in[-1,1]$,
        \begin{align}\label{transversality}
           |P_y(c)-P_{y'}(c)|+|e_y(c)-e_{y'}(c)|\geq m(|P_y(c')-P_{y'}(c')|+|e_y(c')-e_{y'}(c')|).
        \end{align}

     Let
    \begin{align*}
    T_y^{\delta}=\{x=(x',x_d)\in[-1,1]^d:|x'-P_y(x_d)|\leq \delta\}.
    \end{align*}
    Then for any $\epsilon>0$, 
    \begin{align}
        \|\sum_{y\in A}\chi_{T_y^{\delta}}\|_{L^p}\lesssim_{d,\epsilon}\delta^{\frac{1-k}{p'}-\epsilon}\cdot (\sum_{y\in A}|T_y^{\delta}|)^{1/p},\;p'=k+\beta.
    \end{align}
\end{thm}
   \subsection{A multilinear curved Kakeya estimate}
   To prove the curved Kakeya estimate Theorem \ref{curved Kakeya, discretized}, we need a multilinear curved Kakeya estimate of level $k+1$ for all $1\leq k\leq d-1$. There are some results for the case $k=d-1$, one can see Theorem 6.2 of 
   \cite{BCT06} or Theorem 1.3 of \cite{Tao20}. Multilinear curved Kakeya estimates of remaining levels were not mentioned before, to the best of the author's knowledge. Here we can give a multilinear curved Kakeya estimate of all levels with $\log$-loss by the multilinear Kakeya theorem of \cite{CV13} and the multilinear Kakeya to multilinear curved Kakeya argument in Section 5 of \cite{Ben14}. Note that when $k=d-1$, Theorem 1.3 in \cite{Tao20} is stronger than Theorem \ref{multilinear curved Kakeya} below. We believe that for $k<d-1$, the $\log$-loss can be removed like Theorem 1.3 of \cite{Tao20}, but estimates with $\log$-loss are simple and enough in our case.
   
   We state the multilinear Kakeya theorem of \cite{CV13} here for later use.
   \begin{thm}[The multilinear Kakeya theorem,\;\cite{CV13}]
      Let $1\leq k\leq  d-1$ and $\delta>0$. There exists a constant $C_{d,k}$ such that if $\mathcal{T}_1,\dots,\mathcal{T}_{k+1}$ are $k+1$ families of $\delta$-tubes in $\R^d$, then we have
      \begin{align}\label{multilinear Kakeya}
          \int_{\R^d}(\sum_{T_1\in\mathcal{T}_1}\cdots\sum_{T_{k+1}\in\mathcal{T}_{k+1}}\chi_{T_1}\cdots\chi_{T_{k+1}}|e(T_1)\wedge\cdots\wedge e(T_{k+1})|)^{\frac{1}{k}}dx\leq C_{d,k}\cdot\delta^d(\prod_{j=1}^{k+1}\#\mathcal{T}_j)^{\frac{1}{k}}.
      \end{align}
      Here $e(T)\in S^{d-1}$ is the direction of the tube $T$ and $|e(T_1)\wedge u_2\wedge\cdots\wedge e(T_{k+1})|$ is the volume of the parallelepiped spanned by $e(T_1),\dots,e(T_{k+1})$.
   \end{thm}
   The following is the multilinear curved Kakeya estimate we need.
   \begin{thm}\label{multilinear curved Kakeya}
       Suppose $A_1,\dots,A_{k+1}\subset[-1,1]^{d-1}\times [-1,1]^{d-1}$ and $0<\delta<e^{-2}$. Let $A=\cup_{j=1}^{k+1}A_j$ and $\{\gamma_y\}_{y\in A}$ be a collection of smooth curves such that $\gamma_y$ joints $(y_1,-1)$ and $(y_2,1)$ and can be parametrized by its last coordinate
        \begin{align*}
            \gamma_y(c)=(P_y(c),c),\;c\in[-1,1].
        \end{align*} Assume that 
       \begin{align*}
           C=\sup_{y\in A}\|P_y(\cdot)\|_{C^2}<\infty
       \end{align*}
       and
       for some $0<\rho<1$,
       \begin{align*}
           |e_{y^1}(c)\wedge\cdots e_{y^{k+1}}(c)|\geq \rho^k,\;y^{j}\in A_j,\;1\leq j\leq k+1,\;
       \end{align*}
       Let 
       \begin{align*}
    T_y^{\delta}=\{x=(x',x_d)\in[-1,1]^d:|x'-P_y(x_d)|\leq \delta\}.
       \end{align*}
       Then there exists a constant $C'$ only depending on on $d,k$ and $C$ such that
       \begin{align}
          \int(\prod_{j=1}^{k+1}\sum_{y^j\in A_j}\chi_{T_{y_j}^{\delta}}(x))^{\frac{1}{k}}dx\leq (C'\rho^{-1})^{2\log\log(1/\delta)}\delta^d(\prod_{j=1}^{k+1}\#A_j)^{\frac{1}{k}}.
      \end{align}
      One can choose $C'=C_{d,k}(10+C)^{2d}$ for some constant $C_{d,k}$ depending on $d,k$.
   \end{thm}
   
   \begin{proof}
       Our idea is that, by Taylor's formula, in a $\delta^{1/2}$-cube, a curved $\delta$-tube is approximately a straight $\delta$-tube. Therefore, we can apply the multilinear Kakeya estimate (\ref{multilinear Kakeya}) to obtain a recursive formula and iterate.

       Let $C_{Curved}(\delta,\rho)$ denote the smallest constant $C$ such that
       \begin{align*}
          \int(\prod_{j=1}^{k+1}\sum_{y^j\in A_j}\chi_{T_{y_j}^{\delta}}(x))^{\frac{1}{k}}dx\leq C \delta^d(\prod_{j=1}^{k+1}\#A_j)^{\frac{1}{k}}.
      \end{align*}
      We decompose $[-1,1]^d$ into $2\delta^{\frac{1}{2}}$-cubes, say $\{Q\}$. Fix a cube $Q$, after translation we can assume that $Q=[-\delta^{1/2},\delta^{1/2}]^d$. By Taylor's formula, we have
      \begin{align*}
          Q\cap T_y^{\delta}\subset Q\cap\mathcal{N}_{(1+C)\delta}(\gamma_y(0)+\Span\{e_y(0)\}):=T'_y.
      \end{align*}
      $T'_y$ is approximately a tube of size $ (1+C)\delta\times\cdots\times (1+C)\delta\times\delta^{1/2}$.
      Let $\mathcal{T}_j=\{T'_{y^j}:y^j\in A_j\}$, and let $A_j^Q=\{y^j\in A_j:T_{y_j}\cap Q\neq\emptyset\}$. Then after rescaling by a factor $\delta^{-1/2}$, we can apply (\ref{multilinear Kakeya}) at scale $\delta^{1/2}$ and get
      \begin{align*}
          \int_Q(\prod_{j=1}^{k+1}\sum_{y^j\in A_j}\chi_{T_{y_j}^{\delta}}(x))^{\frac{1}{k}}dx\leq \int_Q(\prod_{j=1}^{k+1}\sum_{T_j\in \mathcal{T}_j}\chi_{T_{j}}(x))^{\frac{1}{k}}dx\lesssim C_{d,k} (1+C)^d\rho^{-1}\delta^d(\prod_{j=1}^{k+1}\#A_j^Q)^{\frac{1}{k}}.
      \end{align*}
      Then 
      \begin{align}\label{first step, MCK}
          \int(\prod_{j=1}^{k+1}\sum_{y^j\in A_j}\chi_{T_{y_j}^{\delta}}(x))^{\frac{1}{k}}dx\leq C_{d,k} (1+C)^d\rho^{-1}\delta^d\sum_{Q}(\prod_{j=1}^{k+1}\#A_j^Q)^{\frac{1}{k}}.
      \end{align}
      
      Now let
      \begin{align*}
          T_y^{\delta^{1/2}}=\{x=(x',x_d)\in[-1,1]^d:|x'-P_y(x_d)|\leq (10+C)\delta^{1/2}\}.
      \end{align*}
      By construction, 
      \begin{align*}
          T_y^{\delta}\cap Q\neq\emptyset \implies Q\subset T_y^{\delta^{1/2}}.
      \end{align*}
      Then 
      \begin{align*}
          \#A_j^Q\leq\sum_{y^j\in A_j}\chi_{T_{y^j}^{\delta^{1/2}}}(x),\;x\in Q.
      \end{align*}
      Take the product and then integrate over $Q$, we get 
      \begin{align*}
          (\prod_{j=1}^{k+1}\#A_j^Q)^{\frac{1}{k}}\leq 2^d\delta^{-d/2}\int_{Q}(\prod_{j=1}^{k+1}\sum_{y^j\in A_j}\chi_{T_{y^j}^{\delta^{1/2}}}(x))^{\frac{1}{k}}.
      \end{align*}
      Sum over all $Q$, we get
      \begin{align}\label{second step, MCK}
          \sum_{Q}(\prod_{j=1}^{k+1}\#A_j^Q)^{\frac{1}{k}}\leq 2^d\delta^{-d/2}\int(\prod_{j=1}^{k+1}\sum_{y^j\in A_j}\chi_{T_{y^j}^{\delta^{1/2}}}(x))^{\frac{1}{k}}dx.
      \end{align}
      By definition of $C_{Curved}$, we have 
      \begin{align}\label{inductive hypothesis, MCK}
          \int(\prod_{j=1}^{k+1}\sum_{y^j\in A_j}\chi_{T_{y^j}^{\delta^{1/2}}}(x))^{\frac{1}{k}}dx\leq C_{Curved}((10+C)\delta^{1/2},\rho)(10+C)^d\delta^{d/2}(\prod_{j=1}^{k+1}\#A_j)^{\frac{1}{k}}.
      \end{align}
      By (\ref{first step, MCK}), (\ref{second step, MCK}) and (\ref{inductive hypothesis, MCK}), we get
      \begin{align*}
           \int(\prod_{j=1}^{k+1}\sum_{y^j\in A_j}\chi_{T_{y_j}^{\delta}}(x))^{\frac{1}{k}}dx\leq 2^{-d}C_{d,k}(10+C)^{2d}\rho^{-1}C_{Curved}((10+C)\delta^{1/2},\rho)\delta^d(\prod_{j=1}^{k+1}\#A_j)^{\frac{1}{k}}.
      \end{align*}
      By definition of $C_{Curved}$ again, we get a recursive formula
      \begin{align}
          C_{Curved}(\delta,\rho)\leq C'\rho^{-1}C_{Curved}((10+C)\delta^{1/2},\rho)
      \end{align}
       for some constant $C'$ only depends on $d,k$ and $C$. 
      Iterate at most $2\log\log(1/\delta)$ times (recall that $\delta<e^{-2}$),  we have
      \begin{align}
          C_{Curved}(\delta,\rho)\leq (C'\rho^{-1})^{2\log\log(1/\delta)}.
      \end{align}
   \end{proof}
   \subsection{Proof of Theorem \ref{curved Kakeya, discretized}}
   Now we can prove Theorem \ref{curved Kakeya, discretized} by the Bourgain-Guth argument and induction on scales. Note that if $A$ is a $(\delta,2(k-1)+\beta)$, the transversal condition (\ref{transversality}) implies that
   \begin{align*}
       \{(P_y(c),e_y(c)):y\in A\}
   \end{align*}
   is a $(m\delta, 2(k-1)+\beta)$-set for all $c\in [-1,1]$. This non-concentration condition is easier to handle in the process of induction on scales. It suffices to show the following.
   \begin{thm}\label{version for induction}
       Suppose $A\subset[-1,1]^{d-1}\times[-1,1]^{d-1}$. Let $\{\gamma_y\}_{y\in A}$ be a collection of smooth curves such that $\gamma_y$ joints $(y_1,-1)$ and $(y_2,1)$ and can be parametrized by its last coordinate
        \begin{align*}
            \gamma_y(c)=(P_y(c),c),\;c\in[-1,1].
        \end{align*}
       We assume that $\{\gamma_y\}$ satisfies the followings.
       
       1. (\textbf{Regular condition}) 
       
       There exists $C\geq 1$ such that 
       \begin{align}
           \sup_{y\in A}\|P_{y}(\cdot)\|_{C^2([-1,1])}\leq C.
       \end{align}

        2. (\textbf{Non-concentration condition})
        
        The set 
        \begin{align}\label{non-concentration condition}
            \{(P_y(c),e_y(c))\in\R^{d-1}\times S^{d-1}:y\in A\}
        \end{align}
        is a $(\delta,2(k-1)+\beta)$-set for all $c\in [-1,1]$.
        
        Let
        \begin{align*}
         T_y^{\delta}=\{x=(x',x_d)\in[-1,1]^d:|x'-P_y(x_d)|\leq \delta\}.
         \end{align*}
         Then for any $\epsilon>0$, we have
    \begin{align}
        \|\sum_{y\in A}\chi_{T_y^{\delta}}\|_{L^p}\lesssim_{d,\epsilon,C}\delta^{\frac{1-k}{p'}-\epsilon}\cdot (\sum_{y\in A}|T_y^{\delta}|)^{1/p},\;p'=k+\beta.
    \end{align}
   \end{thm}
   \begin{proof}
       Fix $\epsilon>0$, suppose that for some $\delta_0>0$ and $C_0>0$, we have 
   \begin{align}\label{simplified inductive hypothesis}
        \|\sum_{y\in A}\chi_{T_y^{\tilde{\delta}}}\|_{L^p}\leq C_0\tilde{\delta}^{\frac{1-k}{p'}-\epsilon}\cdot (\sum_{y\in A}|T_y^{\tilde{\delta}}|)^{1/p},\;\tilde{\delta}\geq \delta_0.
    \end{align}
    We need to choose a small enough $\delta_0$ and a large enough $C_0$ to make the induction closed.
    At the end, we shall explain how to choose  $C_0$ and $\delta_0$.
    
    Suppose $\delta<\delta_0$, and (\ref{simplified inductive hypothesis}) holds for $\tilde{\delta} \geq2\delta$. Let $\rho$ be a small constant determined later. We cover $S^{d-1}$ by boundedly overlapping $\rho^{d}/1000d$-caps $\{\tau\}$. The number of such caps is $t\rho^{-d(d-1)}$ for some constant $t$ only depending on $d$.

    Let 
    \begin{align*}
        X=\bigcup_{y\in A}T_{y}^{\delta}.
    \end{align*}
    Given $\tau$ and $x\in X$, we define 
    \begin{align*}
        A_{\tau}(x)=\{y\in A:e_y(x_d)\in\tau\}.
    \end{align*}
    Then we define the significant set
    \begin{align*}
        S'(x)=\{\tau:\sum_{y\in A_{\tau}(x)}\chi_{T_y^{\delta}}(x)\geq \frac{1}{1000\#\{\tau\}}\sum_{y\in A}\chi_{T_y^{\delta}}(x)\}
    \end{align*}
    After a dyadic pigeonholing, we can choose $S(x)\subset S'(x)$ such that
    \begin{align*}
            \sum_{y\in A_{\tau}(x)}\chi_{T_y^{\delta}}(x)/2\leq\sum_{y\in A_{\tau'}(x)}\chi_{T_y^{\delta}}(x)\leq 2\sum_{y\in A_{\tau}(x)}\chi_{T_y^{\delta}}(x),\;\forall\tau,\tau'\in S(x)
        \end{align*}
        and 
        \begin{align*}
            \log(\rho^{-1})\cdot\sum_{\tau\in S(x)}\sum_{y\in A_{\tau}(x)}\chi_{T_y^{\delta}}(x)\gtrsim_d \sum_{y\in A}\chi_{T_y^{\delta}}(x).
        \end{align*}
        
       We want to apply the Bourgain-Guth argument to the set $S(x)$. Roughly speaking, the Bourgain-Guth argument says that given $k$, either most of the caps in $S(x)$ are contained in a small neighborhood of a $k$-dimensional subspace, or most $(k+1)$-tuples are well separated. This idea first appeared in \cite{BG11}.

        To make this dichotomy precise, we need some definitions. 
        For $u_1,\dots,u_{k+1}\in S^{d-1}$, we define $|u_1\wedge u_2\wedge\cdots\wedge u_{k+1}|$ to be the volume of the parallelepiped spanned by $u_1,\dots,u_{k+1}$. For $u\in S^{d-1}$, $H\in G(d,k)$, we define $|H\wedge u|=|e_1\wedge\cdots\wedge e_{k}\wedge u|$ if $e_1,\dots,e_k$ is an orthonormal basis for $H$. Clearly, this definition is independent of the choices of orthonormal basis. We have the following.
   \begin{lem}\label{simplified Bourgain-Guth}
       Suppose $S$ is a collection of boundedly overlapping $\rho^{d}/1000d$-caps on $S^{d-1}$, then at least one of the followingholds.
    
    I. $\exists \tau_1,\dots,\tau_{k+1}\in S$ such that  
       \begin{align*}
           |u_1\wedge\cdots\wedge u_{k+1}|\geq\rho^k/2,\;\forall u_j\in \tau_j,\;1\leq j\leq k+1.
       \end{align*}
       
    II. $\exists H\in G(d,k)$ such that 
       \begin{align*}
           \#\{\tau\in S:\tau\subset\mathcal{N}_{\rho}(H)\}\geq \frac{1}{2}\#S.
       \end{align*}
   \end{lem}
   \begin{proof}
       Suppose $II$ fails.
       Let $e(\tau)\in S^{d-1}$ denote the center of $\tau$, it suffices to find $\tau_1,\cdots,\tau_{k+1}\in S$ such that
       \begin{align}
           |e(\tau_1)\wedge\cdots\wedge e(\tau_{k+1})|\geq \frac{999}{1000}\rho^{k}.
       \end{align}
       Choose any $\tau_1\in S$, let $H_1=span\{e(\tau_1)\}\in G(d,1)$. Since $B$ fails, we have
       \begin{align*}
           \#\{\tau\in S:\tau\subset\mathcal{N}_{\rho}(H_1)\}\leq\frac{1}{2}\#S.
       \end{align*}
       We can find $\tau_2\in S$ such that
       \begin{align*}
           |e(\tau_1)\wedge e(\tau_2)|\geq\rho-\frac{\rho^d}{1000d}\geq\frac{1000d-1}{1000d}\rho.
       \end{align*}
       
       Now let $H_2=span\{e(\tau_1),e(\tau_2)\}\in G(d,2)$. Similarly, since $II$ fails, we can find $\tau_3\in S(U)$ such that
       \begin{align*}
           |e(\tau_1)\wedge e(\tau_2)\wedge e(\tau_3)|\geq\frac{1000d-1}{1000d}\rho^2-\frac{\rho^d}{1000}\geq \frac{1000d-2}{1000d}\rho^2.
       \end{align*}
       We continue in this way. After $k$ steps, we can find $\tau_1,\cdots,\tau_{k+1}\in S(U)$ such that
       \begin{align*}
           |e(\tau_1)\wedge\cdots\wedge e(\tau_{k+1})|\geq\frac{999}{1000}\rho^{k}.
       \end{align*}
   \end{proof}
   
   We apply this lemma to $S=S(x)$, and say that $x$ is broad if case $I$ holds, and $x$ is narrow if case $I$ fails. Let $X_b$ denote the set of broad points, and $X_n$ denote the set of narrow points. We have 
   \begin{align*}
       X=X_{b}\bigcup X_{n}
   \end{align*}
   and
   \begin{align}
       \|\sum_{y\in A}\chi_{T_y^{\delta}}\|_{L^p}^p=\|\sum_{y\in A}\chi_{T_y^{\delta}}\|_{L^p(X_b)}^p+\|\sum_{y\in A}\chi_{T_y^{\delta}}\|_{L^p(X_n)}^p.
   \end{align}

    We first estimate $\|\sum_{y\in A}\chi_{T_y^{\delta}}\|_{L^p(X_b)}^p$. We define the transversal set
    \begin{align*}
        T^{k+1}=\{(\tau_1,\dots,\tau_{k+1}):|u_1\wedge\cdots\wedge u_{k+1}|\geq\rho^k/2,\;u_j\in \tau_j,\;1\leq j\leq k+1\}.
    \end{align*}
    By construction, we have
    \begin{align*}
        X_b=\bigcup_{(\tau_1,\dots,\tau_{k+1})\in T^{k+1}}\{x\in X_b:\tau_1,\dots,\tau_{k+1}\in S(x)\}.
    \end{align*}
    By pigeonholing, there exists a fixed $(k+1)$ tuples $(\tau_1,\dots,\tau_{k+1})\in T_{k+1}$ such that if 
    \begin{align*}
        X_b'=\{x\in X_b:\tau_1,\dots,\tau_{k+1}\in S(x)\},
    \end{align*}
    then 
    \begin{align}
        \|\sum_{y\in A}\chi_{T_y^{\delta}}\|_{L^p(X_b)}^p\lesssim_d\rho^{-d^3}\|\sum_{y\in A}\chi_{T_y^{\delta}}\|_{L^p(X_b')}^p.
    \end{align}
    For each $x\in X_b'$, we have 
    \begin{align}
        \sum_{y\in A}\chi_{T_y^{\delta}}(x)\lesssim_d\rho^{-d^2}\cdot\prod_{j=1}^{k+1}(\sum_{y^j\in A_{\tau_j}(x)}\chi_{T_{y^j}^{\delta}}(x))^{\frac{1}{k+1}}.
    \end{align}
    Integrate over $X_b'$, we get
    \begin{align}\label{simplified broad}
        \|\sum_{y\in A}\chi_{T_y^{\delta}}\|_{L^p(X_b)}^p\lesssim_d\rho^{-d^3}\|\sum_{y\in A}\chi_{T_y^{\delta}}\|_{L^p(X_b')}^p\lesssim_d\rho^{-10d^3}\int\prod_{j=1}^{k+1}(\sum_{y^j\in A_{\tau_j}(x)}\chi_{T_{y^j}^{\delta}}(x))^{\frac{p}{k+1}}dx.
    \end{align}
    Note that by Taylor's formula,  
     \begin{align}
         |e_y(c)-e_y(c')|\leq 2C|c-c'|,\;\forall c,c',y.
     \end{align}
    We cover $[-1,1]^{d}$ by boundedly overlapping $\frac{\rho^{d}}{1000dC}$-cubes $\{Q\}$, let $c_{Q}$ denote the last coordinate of the center of $Q$. Fix $Q$, let $A_{\tau_j}(Q)=\{y^j\in A:e_{y^j}(c_{Q})\in2\tau_j\}$, $1\leq j\leq k+1$. We have 
     \begin{align*}
         e_{y^j}(x_d)\in\tau_j,x\in Q\implies e_{y^j}(c_Q)\in 2\tau_j.
     \end{align*}
    Then
   \begin{align*}
       \int\prod_{j=1}^{k+1}(\sum_{y_j\in A_{\tau_j}(x)}\chi_{T_{y_j}^{\delta}}(x))^{\frac{p}{k+1}}dx\leq \sum_{Q}\int_{Q}(\prod_{j=1}^{k+1}\sum_{\substack{y^j\in A_{\tau_j}(Q)}}\chi_{T_{y^j}^{\delta}}(x))^{\frac{p}{k+1}}dx.
   \end{align*}
   Also by Taylor's formula and definition of $T^{k+1}$, when $|c-c_{Q}|\leq\frac{\rho^{d}}{1000dC}$,
   \begin{align}
       e_{y^j}(c_{Q})\in2\tau_j,\;1\leq j\leq k+1\implies |e_{y^1}(c)\wedge\cdots\wedge e_{y^{k+1}}(c)|\geq\rho^k/4.
   \end{align}
   Then we can apply Theorem \ref{multilinear curved Kakeya} and conclude that
   \begin{align}
       \int_{Q}(\prod_{j=1}^{k+1}\sum_{\substack{y^j\in A_{\tau_j}(Q)}}\chi_{T_{y^j}^{\delta}}(x))^{\frac{p}{k+1}}dx\leq (100C'\rho^{-1})^{2\log\log(1/\delta)}\delta^d(\#A)^{p
       },
   \end{align}
   here $C'$ is a constant only depending on on $d,k$ and $C$.
   The non-concentration condition (\ref{non-concentration condition}) implies that 
    \begin{align}
        \# A\leq \delta^{-2(k-1)-\beta}.
    \end{align}
    Recall that $\#\{Q\}\lesssim_{d,C} \rho^{-d^2}$, we sum over all $Q$ and get
   \begin{align}\label{simplified broad final}
       \|\sum_{y\in A}\chi_{T_y^{\delta}}\|_{L^p(X_b)}^p\leq\overline{C}\rho^{-100d^{3}}(100C'\rho^{-1})^{2\log\log(1/\delta)}\delta^{\frac{p(1-k)}{p'}}\sum_{y\in A}|T_y^{\delta}|,
   \end{align}
   where $\overline{C}$ is a constant independent of $\rho,\delta$.
   
    Now we turn to $\|\sum_{y\in A}\chi_{T_y^{\delta}}\|_{L^p(X_n)}^p$. For each $x\in X_n$, there exists $H=H(x)\in G(d,k)$ such that
    \begin{align}
        \sum_{y\in A}\chi_{T_y^{\delta}}(x)\lesssim_d\log(\rho^{-1})\cdot\sum_{\tau\subset\mathcal{N}_{\rho}(H)}\sum_{y\in A_{\tau}(x)}\chi_{T_y^{\delta}}(x)\lesssim\sum_{e_{y}(x_d)\in S^{d-1}\cap \mathcal{N}_{\rho}(H)}\chi_{T_y^{\delta}}(x).
    \end{align}
    We cover $S^{d-1}$ by boundedly overlapping $\rho$-caps $\{\tilde{\tau}\}$.
    Since $H$ has dimension $k$, 
   \begin{align*}
       \#\{\tilde{\tau}:\tilde{\tau}\cap\mathcal{N}_{\rho}(H)\neq\emptyset\}\lesssim_{d,k}\rho^{1-k}.
   \end{align*}
   Then by H\"older's inequality,
   \begin{align}
       (\sum_{y\in A}\chi_{T_y^{\delta}}(x))^p\lesssim_{d,k}\log(\rho^{-1})^p\cdot\rho^{(1-k)(p-1)}\cdot\sum_{\tilde{\tau}}(\sum_{y\in A_{\tilde{\tau}}(x)}\chi_{T_y^{\delta}}(x))^p,\;x\in X_n.
   \end{align}
   Integrate over $X_n$, we get
   \begin{align}\label{simplified narrow}
        \|\sum_{y\in A}\chi_{T_y^{\delta}}\|^p_{L^p(X_n)}\lesssim_{d,k}\log(\rho^{-1})^p\cdot\rho^{(1-k)(p-1)}\cdot\int\sum_{\tilde{\tau}}(\sum_{y\in A_{\tilde{\tau}}(x)}\chi_{T_y^{\delta}}(x))^pdx.
   \end{align}
    We cover $[-1,1]$ by boundedly overlapping $\frac{\rho}{1000C}$-intervals $\{J\}$, and let $c_{J}$ be the center of $J$. Then
   \begin{align}
       \int\sum_{\tilde{\tau}}(\sum_{y\in A_{\tilde{\tau}}(x)}\chi_{T_y^{\delta}}(x))^pdx\lesssim \sum_{J}\int_{S_{J}}\sum_{\tilde{\tau}}(\sum_{y\in A_{\tilde{\tau}}(x)}\chi_{T_y^{\delta}}(x))^pdx,
   \end{align}
   where 
   \begin{align*}
       S_{J}=[-1,1]^{d-1}\times J.
   \end{align*}
    Also, by Taylor's formula,
     \begin{align}
         y\in A_{\tilde{\tau}}(x),\;x\in S_{J}\implies y\in A_{\tilde{\tau}}(J),
     \end{align}
    where 
    \begin{align*}
        A_{\tilde{\tau}}(J)=\{y\in A:e_{y}(c_{J})\in2\tilde{\tau}\}.
    \end{align*}
    Then
    \begin{align}
        \sum_{J}\sum_{\tilde{\tau}}\int_{S_{J}}(\sum_{y\in A_{\tilde{\tau}}(x)}\chi_{T_y^{\delta}}(x))^pdx\leq \sum_{J}\sum_{\tilde{\tau}}\int_{S_{J}}(\sum_{y\in A_{\tilde{\tau}}(J)}\chi_{T_y^{\delta}}(x))^pdx.
    \end{align}
    Fix $J$ and $\tilde{\tau}$. Let $\mathcal{T}_{\tilde{\tau}}$ be a collection of boundedly overlapping $\rho^{2}$-tubes that point in direction $e(\tilde{\tau})$, the center of $\tilde{\tau}$. 
   Note that by Taylor's formula,
   \begin{align}
       e_y(c_{J})\in2\tilde{\tau}\implies e_y(c)\in 3\tilde{\tau},\;c\in J\implies\exists T\in \mathcal{T}_{\tilde{\tau}},\;T_{y}^{\delta} \cap S_{J}\subset T.
   \end{align}
   For each $T\in \mathcal{T}_{\tilde{\tau}}$,  we define 
   \begin{align*}
       A_T=\{y\in A:T_{y}^{\delta} \cap S_{J}\subset T,\;e_y(c_{J})\in2\tilde{\tau}\}.
   \end{align*}
   Then
   \begin{align}
       \int_{S_{J}}(\sum_{y\in A_{\tilde{\tau}}(J)}\chi_{T_y^{\delta}}(x))^pdx\lesssim \sum_{T\in\mathcal{T}_{\tilde{\tau}}}\int_{S_{J}}(\sum_{\substack{y\in A_T}}\chi_{T_{y}^{\delta}}(x))^pdx.
   \end{align}
   Thus 
   \begin{align}
       \sum_{\tilde{\tau}}\int_{S_{J}}(\sum_{y\in A_{\tilde{\tau}}(J)}\chi_{T_y^{\delta}}(x))^pdx\lesssim\sum_{\tilde{\tau}}\sum_{T\in\mathcal{T}_{\tilde{\tau}}}\int_{S_{J}}(\sum_{\substack{y\in A_T}}\chi_{T_{y}^{\delta}}(x)))^pdx.
   \end{align}
   
       Now we fix $T\in\mathcal{T}_{\tilde{\tau}}$. After translation, we assume that $T$ and $S_J$ are centered at $0$. We consider the linear change of coordinates $\psi_{\rho}$ that preserves $\R^{d-1}\times\{0\}$ and sends $e(\tilde{\tau})$ to $e_d$. We have 
       $|\psi_{\rho}(v)|\sim_C |v|,\;\forall v$ and $|\det \psi_{\rho}|\sim_C 1 $. 
       Thus, after multiplication by a constant depending on $C$, we can assume that $e(\tilde{\tau})=e_d$.

       Suppose $J=[-\frac{\rho}{2000C},\frac{\rho}{2000C}]$. We define $E_{\rho}(y)=(P_{y}(-\frac{\rho}{2000C}),P_{y}(\frac{\rho}{2000C}))$ and consider the linear map $A_{\rho}$ given by
   \begin{align*}
       e_{j}\mapsto\frac{1}{\rho^2}e_{j},\;1\leq j\leq d-1;\;e_d\mapsto \frac{2000C}{\rho}e_d.
    \end{align*}
    Then we get a new collection of smooth curves
   \begin{align*}
       \tilde{\gamma}_z=A_{\rho}(S_{J}\cap\gamma_y),\;z=\frac{1}{\rho^{2}}E_{\rho}(y),\;y\in A_T.
   \end{align*}
   By construction, if $T_{y}^{\delta} \cap S_{J}\subset T$ , $A_{\rho}(T_{y}^{\delta} \cap S_{J})$ is a $\frac{\delta}{\rho^{2}}$-tube centered at the curve $\tilde{\gamma}_z$ and 
   \begin{align*}
       A_{\rho}(T^{\delta}_{y} \cap S_{J})\subset A_\rho(T)\subset[-1,1]^d.
   \end{align*} 
   We can reparametrize $\tilde{\gamma}_z$ by its last coordinate
   \begin{align*}
       \tilde{\gamma}_z(c)=(\tilde{P}_{z}(c),c),\;c\in [-1,1] ,
   \end{align*}
   where 
   \begin{align*}
       \tilde{P}_{z}(c)=\frac{1}{\rho^{2}}P_{y}(\frac{\rho}{2000C}c),\;z=\frac{1}{\rho^{2}}E_{\rho}(y).
   \end{align*}
    
   It is clear that the regular condition holds with the same $C$ and the set
   \begin{align*}
       \{(\tilde{P}_{z}(c), \tilde{e}_{z}(c))\in\R^{d-1}\times S^{d-1}:z\in A_{\rho}(A_T)\}
   \end{align*}
   is a $(M_C\frac{\delta}{\rho},2(k-1)+\beta)$-set for some constant $M_C$ depending only on $C$ and $d$.
   By our induction hypothesis (\ref{simplified inductive hypothesis}), we have 
   \begin{align}
      \int_{S_{J}}(\sum_{\substack{y\in A_T}}\chi_{T_{y}^{\delta}}(x))^pdx\lesssim_{d,C} C_0^p\rho^{-\frac{p(1-k)}{p'}+p\epsilon}(M_C\delta)^{\frac{p(1-k)}{p'}-p\epsilon}\cdot |J|\cdot\sum_{\substack{y\in A_T}}|T_y^{\delta}|.
   \end{align}
   Thus 
  \begin{align}
       \|\sum_{y\in A}\chi_{T_y^{\delta}}\|^p_{L^p(X_n)}\lesssim_{d,k,C} C_0^p\log(\rho^{-1})^p\rho^{p\epsilon}\delta^{\frac{p(1-k)}{p'}-p\epsilon}\cdot\sum_{\tilde{\tau}}\sum_{T\in\mathcal{T}_{\tilde{\tau}}}\sum_{\substack{y\in A_T}}|T_y^{\delta}|.
   \end{align}
   Recall that if $e_y(c_{J})\in 2\tilde{\tau}$, we have
   $e_y(c)\in 3\tilde{\tau}$, $\forall|c-c_{J}|\leq\frac{\rho}{2000C}$. Then since $\{\tilde{\tau}\}$ is boundedly overlapping, we have 
   \begin{align}\label{simplified boundedly overlapping}
       \sum_{\tilde{\tau}}\sum_{T\in\mathcal{T}_{\tilde{\tau}}}\sum_{\substack{y\in A_T}}|T_y^{\delta}|\lesssim_d\sum_{y\in A}|T_y^{\delta}|.
   \end{align}
   Thus for some constant $\hat{C}$ independent with $\delta,\rho$, we have 
   \begin{align}\label{simplified narrow final}
       \|\sum_{y\in A}\chi_{T_y^{\delta}}\|^p_{L^p(X_n)}&\leq \hat{C}C_0^p\log(\rho^{-1})^p\rho^{p\epsilon}\cdot\delta^{\frac{p(1-k)}{p'}-p\epsilon}\sum_{\substack{y\in A}}|T_y^{\delta}|.
   \end{align}

   Thus by (\ref{simplified broad final}) and (\ref{simplified narrow final}), we have 
   \begin{align}\label{final}
       &\|\sum_{y\in A}\chi_{T_y^{\delta}}\|_{L^p}^p\\
       \leq& \left(\overline{C}\rho^{-100d^{3}}(100C'\rho^{-1})^{2\log\log(1/\delta)}\delta^{p\epsilon}+\hat{C}C_0^p\log(\rho^{-1})^p\rho^{p\epsilon}\right)\delta^{\frac{p(1-k)}{p'}-p\epsilon}\sum_{y\in A}|T_y^{\delta}|.\nonumber
   \end{align}
   First, we choose $\rho$ small enough such that
   \begin{align}
       \hat{C}\log(\rho^{-1})^p\rho^{p\epsilon}\leq \frac{1}{2}.
   \end{align}
   Then we take $\delta_0$ small enough such that 
   \begin{align}
       \overline{C}\rho^{-100d^{3}}(100C'\rho^{-1})^{2\log\log(1/\delta)}\delta^{p\epsilon}\leq\overline{C}\rho^{-100d^{3}}(100C'\rho^{-1})^{2\log\log(1/\delta_0)}\delta_0^{p\epsilon},\;\delta\leq\delta_0.
   \end{align}
   Note that we have a trivial bound
   \begin{align*}
       \|\sum_{y\in A}\chi_{T_y^{\delta}}\|^p_{L^p}\leq\#A^p\leq\delta^{-d}\cdot\delta^{\frac{p(1-k)}{p'}}\sum_{y\in A}|T_{y}^{\delta}|.
   \end{align*}
   Hence, it suffices to take 
   \begin{align}
       C_0^p\geq \max\{2\overline{C}\rho^{-100d^{3}}(100C'\rho^{-1})^{2\log\log(1/\delta_0)}\delta_0^{p\epsilon},\delta_0^{-d}\}.
   \end{align}
   \end{proof}

   \section{Further Remarks}
   \subsection{Sharpness of Theorem \ref{Hausdorff dimension, unions of geodesics}}
   In this section, we will give explicit examples to show that Theorem \ref{Hausdorff dimension, unions of geodesics} is sharp for complete manifolds with constant sectional curvature. 

   \begin{thm}
       Suppose $M$ is a $d$-dimensional complete manifold with constant sectional curvature, and $\pi$ is the projection from $SM$ to $M$. Given an integer $1\leq k\leq d-1$ and $\beta\in[0,1]$, there exists an invariant set $E\subset SM$ with $\dH E=2(k-1)+\beta+1$ and $\dH\pi(E)=k+\beta$.
   \end{thm}
   \begin{proof}
       It is well known that the universal covering of a $d$-dimensional complete manifold with constant sectional curvature is isometric to one of $\R^d$, $\mathbb{H}^d$, and $S^d$, and the covering map is a local isometry. Thus it suffices to consider $M=\R^d$, $\mathbb{H}^d$ and $S^d$.
   
   In the case $M=\R^d$, we can just take the set of all lines contained in a $\beta$-dimensional union of parallel affine $k$-planes. Similarly, in the case $M=\mathbb{H}^d$, we take the set of all geodesics contained in a $\beta$-dimensional union of parallel affine $k$-planes that are orthogonal to the boundary $\{x_d=0\}$. 

   The case $M=S^d$ is slightly different from the other two cases, since there are no parallel $k$-dimensional totally geodesic submanifolds when $k>(d-1)/2$. Our idea is choose a collection of $k$-dimensional totally submanifolds such that their overlapping is as small as possible, i.e, any two of them intersect at a fixed $(k-1)$-dimensional submanifold.

    We take
   \begin{align*}
       S^k=\{x=(x_i)_{i=1}^{d+1}\in S^d:x_{k+2}=\cdots=x_{d+1}=0\}.
   \end{align*}
   
   We take a $\beta$-dimensional set $A\subset [-\pi/6,\pi/6]$.
   For each $\theta\in A$, we define the orthogonal map $g_{\theta}\in O(d-k+1)$ such that
   \begin{align*}
       g_{\theta}(1,0,\dots,0)&=(\cos\theta,\sin\theta,0,\dots,0),\\
       g_{\theta}(0,1,\dots,0)&=(-\sin\theta,\cos\theta,0,\dots,0),\\
       g_{\theta}(0,0,x_{k+3},\dots,x_{d+1})&=(0,0,x_{k+3},\dots,x_{d+1}).
   \end{align*}
   It is easy to check that
   \begin{align*}
       \|g_{\theta}-g_{\theta'}\|_{\R^{(d-k+1)\times(d-k+1)}}\sim|\theta-\theta'|,\;\theta,\;\theta'\in [-\pi/6,\pi/6].
   \end{align*}
   Then if $O_A=\{g_{\theta}\}_{\theta\in A}$ we have $\dH O_A=\dH A=\beta$. Let $S^k_{\theta}:=(Id_{k}\oplus g_{\theta})(S^k)$. $S^k_{\theta}$ is a $k$-dimensional totally geodesic submanifold, and 
   \begin{align*}
       S^k_{\theta}\bigcap S^k_{\theta'}=S^{k-1},\;\forall \theta\neq \theta'.
   \end{align*}

   We take 
   \begin{align*}
       E=\bigcup_{\theta\in A}S(S^k_{\theta})\subset SM.
   \end{align*}
   $E$ is invariant. By construction,
   \begin{align*}
       \pi(E)=\bigcup_{\theta\in A}S^k_{\theta}.
   \end{align*}
   
   It remains to show that
   \begin{align}
       \dH E=2(k-1)+\beta+1
   \end{align}
   and 
   \begin{align}
       \dH\pi(E)=k+\beta.
   \end{align}
   
   We first compute the Hausdorff dimension of $\pi(E)$. Note that the map
   \begin{align*}
       \phi:(x,g)\in S^d\times O(d-k+1)\mapsto (Id_k\oplus g)(x)\in S^d.
   \end{align*}
   is Lipschitz, then
   \begin{align*}
       \dH\pi(E)=\dH\phi(S^{k}\times O_A)\leq \dH(S^{k}\times O_A)= k+\beta.
   \end{align*}
   For the inverse, we choose a point $p\in S^{k-1}= \bigcap_{e\in A}S^k_e$ and consider the exponential map 
   \begin{align*}
    \exp_p:B(0,\delta)\subset T_pS^d\to SB_{\delta}(p).
   \end{align*}
   The exponential map $\exp_p$ is a diffeomorphism when $\delta<\pi$. Note  $S_{\theta}^k$ are totally geodesic, we have 
   \begin{align*}
       \exp_p(t\cdot\bigcup_{e\in A }S_p(S^k_{\theta}))\subset \pi(E),\;\forall t\in \R.
   \end{align*}
   It is not hard to check that
   \begin{align*}
       \dH (\bigcup_{e\in A }S_p(S^k_{\theta}))=k-1+\beta.
   \end{align*}
    Then
   \begin{align*}
       \dH(\{tv:t\in(0,\delta),\;v\in\bigcup_{e\in A }S_p(S^k_{\theta})\})\geq \dH (\bigcup_{e\in A }S_p(S^k_{\theta}))+1=k+\beta.
   \end{align*}
   Thus 
   \begin{align*}
       \dH\pi(E)\geq k+\beta.
   \end{align*}

    Now we turn to the Hausdorff dimension of $E$. Note that the projection $\pi|_E$ is Lipschitz and for $p\in\pi(E)$, by construction of $E$,
    \begin{align*}
        \dH\pi|_E^{-1}(\{p\})\geq k-1.
    \end{align*}
    By Lemma in \cite{KM75}, we have 
    \begin{align*}
        \dH E\geq k-1+\dH \pi(E)=2(k-1)+\beta+1.
    \end{align*}
    On the other hand, we consider the map
    \begin{align*}
        \psi:S(S^k)\times A\to E,\;\psi((p,v),\theta)=(g_\theta (p), g_\theta (v)).
    \end{align*}
    It is clear that $\psi$ is Lipschitz and then 
    \begin{align*}
        \dH E\leq \dH (S(S^k)\times A)=2(k-1)+\beta+1.
    \end{align*}
    \end{proof}
    \subsection{Conjecture for manifolds with non-constant sectional curvature}
    When $M$ is a complete manifold with non-constant sectional curvature and dimension $\geq 3$, it is natural to ask whether Theorem \ref{Hausdorff dimension, unions of geodesics} is still sharp. We believe that Theorem \ref{Hausdorff dimension, unions of geodesics} can be improved in this case and give the following conjecture.
    \begin{conj}\label{Hausdorff dimension, unions of geodesics, non-constant sectional curvature}
        Suppose $d\geq 3$ and
        $M$ is a $d$-dimensional complete Riemannian manifold with \textbf{non-constant sectional curvature}. Let $1\leq k\leq d-1$ be an integer and $\beta\in[0,1]$. Then there exists $\epsilon=\epsilon(d,k,\beta, M)>0$ such that if $E\subset SM$ is an invariant set with $\dH E\geq 2(k-1)+1+\beta$, then 
        \begin{align*}
            \dH\pi(E)\geq k+\beta+\epsilon.
        \end{align*}
    \end{conj}

\bibliographystyle{abbrv}
    \bibliography{Unions_of_geodesics}
\end{document}